\documentclass{ws-jaa}
\def\G{\Gamma}
\def\Z{\mathbb{Z}}

\begin{document}

\bibliographystyle{plain}

\title{WHEN IDEAL-BASED ZERO-DIVISOR GRAPHS ARE COMPLEMENTED OR UNIQUELY COMPLEMENTED}

\author{\footnotesize Jesse Gerald Smith Jr.}

\address{Division of Mathematics and Computer Science \\ Maryville College \\ 502 East Lamar Alexander Parkway, Maryville, TN, 37804 \\ e-mail: jesse.smith@maryvillecollege.edu}

%

\maketitle


\begin{abstract}
Let $R$ be a commutative ring with nonzero identity and $I$ a proper ideal of $R$. The {\it ideal-based zero-divisor graph} of $R$ with respect to the ideal $I$, denoted by $\Gamma_I(R)$, is the graph on vertices $\{x \in R\setminus I \mid xy\in I$ for some $y\in R\setminus I \}$, where distinct vertices $x$ and $y$ are adjacent  if and only if $xy\in I$.  In this paper,we give a complete classification of when an ideal-based zero-divisor graph of a commutative ring is complemented or uniquely complemented based on the total quotient ring of $R/I$.
\end{abstract}


\ccode{2000 Mathematics Subject Classification: 13A99,05C99}
\section{Preliminaries}

Let $R$ be a commutative ring with nonzero identity, $I$ a proper ideal of $R$, and $Z(R)$ the set of zero-divisors of $R$. Throughout this paper, a {\it graph} will always be a simple graph, i.e., an undirected graph without multiple edges or loops. In 1988, I. Beck used zero-divisors to produce a graph given a ring $R$ \cite{Beck1988}; he was interested in colorings of these graphs. In 1999, D. F. Anderson and P. S. Livingston modified Beck's definition to the following \cite{AndLiv,livThesis}; the {\it zero-divisor graph} of $R$, denoted by $\G(R)$, is the graph on the vertex set $Z(R)^*=Z(R)\setminus\{0\}$, where two distinct vertices $x$ and $y$ are adjacent if and only if $xy=0$.   In 2001, S. P. Redmond gave the following definition (\cite{redDis} and \cite{Redmond1}) as a generalization of the zero-divisor graph; the graph on vertex set $\{x \in R\setminus I \mid xy\in I$ for some $y\in R\setminus I \}$, where distinct vertices $x$ and $y$ are adjacent  if and only if $xy\in I$. This is called the {\it ideal-based zero-divisor graph} of $R$ with respect to the ideal $I$, denoted by $\G_I(R)$. Note that $\G_I(R)$ and $\G(R/I)$ are non-empty if and only if $I$ is not a prime ideal of $R$.

Recall that a ring $R$ is von Neumann regular if for every $x\in R$, there exists a $y\in R$ such that $x=xyx$. In \cite{DFAetal2003}, the authors find a connection between a ring being von Neumann regular and a graph property called complemented. They define $a\sim b$ if $a$ and $b$ are not adjacent, yet they are adjacent to exactly the same vertices of $G$. Given distinct vertices $a$ and $b$ of a graph $G$, we say that the vertices are {\it orthogonal}, denoted $a\perp b$, if $a$ and $b$ are adjacent and there is no vertex adjacent to both $a$ and $b$. Notice that $a\perp b$ if and only if $a$ and $b$ are adjacent and the edge $a-b$ is not part of triangle (a 3-cycle) in $G$. A graph $G$ is called {\it complemented} if given any vertex $a$ of $G$, there exists a vertex $b$ of $G$ such that $a\perp b$. A graph $G$ is {\it uniquely complemented} if it is complemented and $a\perp b$ and $a\perp c$ imply that $a\sim c$. The preceding relations and definitions are from \cite{DFAetal2003} and \cite{rs2002}. In \cite[Theorem 3.5]{DFAetal2003}, the authors show that for a reduced ring $R$, $\G(R)$ is uniquely complemented if and only if $\G(R)$ is complemented, if and only if $T(R)$ is von Neumann regular. In this paper, we extend this result to $\G_I(R)$.

Throughout this paper, $R$ will be a commutative ring with nonzero identity, $Z(R)$ its set of zero-divisors, $nil(R)$ its ideal of nilpotent elements, and total quotient ring $T(R)=R_S$, where $S=R\setminus \{0\}$. Given an ideal $I$ of $R$, we define $\sqrt{I}=\{r\in R\mid r^k \in I\text{ for some } k\in \mathbb{N}\}$. A ring $R$ is reduced if $nil(R)=\sqrt{\{0\}}=\{0\}$. Notice that $R/I$ is reduced if and only if $\sqrt{I}=I$. An ideal $I$ is a radical ideal if $\sqrt{I}=I$. Let $\Z$ and $\Z_n$ denote the integers and the integers modulo $n$, respectively. We will also use the well-known result that $|Z(R)|=2$ if and only if $R/I\cong \Z_4$ or $\Z_2[X]/(X^2)$. We will denote the set of vertices of a graph $G$ by $V(G)$. In this paper, we will also use that $|V(\G_I(R))|=|I||V(\Gamma(R/I)|$ \cite[Corollary 2.7]{Redmond1}.We say that a graph is complete on $n$ vectrices, denoted by $K^n$, if it is a graph on $n$ vectiecs in which each vertex is connected to all other vertices.

\section{ When $\G_I(R)$ is complemented or uniquely complemented}

We consider the situation in two cases: either $I$ is a radical ideal of $R$ or $I$ is a non-radical ideal of $R$.

\begin{proposition}
Let $R$ be a commutative ring with nonzero identity and $I$ a nonzero, non-radical ideal of $R$. If $|V(\Gamma(R/I))|\geq 2$, then $\Gamma_I(R)$ is not complemented.
\label{ch3:compPROP1}
\end{proposition}

\begin{proof} Since $I\neq \sqrt{I}$, there exists an $r\in R\setminus I$ such that $r^2\in I$.  Then $r\in V(\Gamma_I(R))$. We claim that $r$ has no complement in $\Gamma_I(R)$. Let $s$ be any vertex of $\Gamma_I(R)$ adjacent to $r$; so $rs\in I$. Notice that $r\neq s$ as they are distinct adjacent vertices of $\G_I(R)$. Then there are two possibilities: (1) there exists an $i\in I$ such that $s=r+i$ or (2) $s\neq r+i$ for all $i\in I$.

Case (1): Assume there exists an $i\in I$ such that $s=r+i$. Then $r+I=s+I$ in $R/I$. Since $|V(\Gamma(R/I))|\geq 2$ and $\Gamma(R/I)$ is connected, there exists a vertex $t+I$ adjacent to $r+I=s+I$ in $\Gamma(R/I)$. Notice that $t,r,s=r+i$ are all distinct vertices of $\Gamma_I(R)$ that are mutually adjacent. Thus the edge $r-s$ is part of a triangle in $\Gamma_I(R)$; so $s$ is not a complement of $r$ in $\Gamma_I(R)$.

Case (2): Assume $s\neq r+i$ for all $i\in I$. Since $I$ is non-zero, choose $0\neq i\in I$. Then the vertices $s,r,r+i$ are distinct mutually adjacent vertices of $\Gamma_I(R)$. Thus the edge $r-s$ is part of a triangle in $\Gamma_I(R)$; so, as before, $s$ is not a complement of $r$ in $\Gamma_I(R)$.

Thus no vertex adjacent to $r$ is a complement of $r$; so $\Gamma_I(R)$ is not complemented. \end{proof}

\begin{lemma} Let $R$ be a commutative ring with nonzero identity and $I$ an ideal of $R$. If $\G(R/I)\cong K^1$, then $\G_I(R)\cong K^{|I|}$.
\label{completelemma}
\end{lemma}

\begin{proof} $|V(\G(R/I))|=1$ if and only $|Z(R/I)|=2$, if and only if $R/I\cong \Z_4$ or $\Z_2[X]/(X^2)$. Thus $V(\G(R/I))=\{a+I\}$, where $a^2\in I$. Then $V(\G_I(R))=\{a+i\}_{i\in I}$. Notice that $(a+i)(a+j)\in I$ for all $i,j\in I$. Moreover $|V(\G_I(R))|=|I||V(\G(R/I))|=|I|\cdot 1=|I|$. Thus $\G_I(R)\cong K^{|I|}$.\label{completeGI} \end{proof}

\begin{theorem} \label{notradicalK2} Let $R$ be a commutative ring with nonzero identity and $I$ a non-radical ideal of $R$. Then $\Gamma_I(R)$ is complemented if and only $\Gamma_I(R)\cong K^2$. 
\end{theorem}

\begin{proof}
The ``$\Leftarrow$'' implication is clear.

Conversely assume that $\Gamma_I(R)$ is complemented. Then $|V(\Gamma(R/I))|\leq 1$ by Proposition \ref{ch3:compPROP1}. Since $I$ is not prime (as it is non-radical), it follows that $|V(\Gamma(R/I))|=1$. Thus $\Gamma_I(R)\cong K^{|I|}$ by Lemma \ref{completelemma}. Since the only complemented complete graph is $K^2$, it follows that $|I|=2$ and $\Gamma_I(R)\cong K^2$.
\end{proof}

Notice that if $|V(\Gamma(R/I))|=1$, then $R/I\cong \mathbb{Z}_4$ or $\mathbb{Z}_2[X]/(X^2)$; so $ \sqrt{I}\neq I$. Moreover, in this case, $\Gamma_I(R)$ is complemented if and only if $|I|=2$ by the preceding theorem. Thus it remains to investigate the case when $|V(\Gamma(R/I))|\geq 2$.

\begin{theorem}
\label{comptheorem}
Let $R$ be a commutative ring with nonzero identity and $I$ a nonzero, non-prime ideal of $R$. Then $\Gamma_I(R)$ is complemented and $|V(\Gamma(R/I))|\geq 2$ if and only if $\Gamma(R/I)$ is complemented and $\sqrt{I}=I$.
\end{theorem}

\begin{proof}
``$\Rightarrow$'' Assume that $\Gamma_I(R)$ is complemented and $|V(\Gamma(R/I))|\geq 2$. Then $I=\sqrt{I}$ by Proposition \ref{ch3:compPROP1}. So it remains to show that $\Gamma(R/I)$ is complemented.
Let $r+I$ be vertex of $\Gamma(R/I)$. Then $r$ is a vertex of $\Gamma_I(R)$. By assumption, $\Gamma_I(R)$ is complemented; so there exists a vertex $s$ of $\Gamma_I(R)$ such that $r\perp s$. We first show that $r+I\neq s+I$. Assume to the contrary; then $r-s=i\in I$. Thus $r(r-s)=ri\in I$. Since $r\perp s$, then $rs\in I$. Hence $r^2=ri+rs\in I$, and thus $r\in I$ since $\sqrt{I}=I$. This is a contradiction since $r+I\neq I$. Thus $r+I\neq s+I$. Since $r\perp s$ in $\Gamma_I(R)$ and $r+I\neq s+I$, it follows that $r+I$ is adjacent to $s+I$ in $\Gamma(R/I)$. It now remains only to show there is no other vertex in $\Gamma(R/I)$ adjacent to both of these. Assume to the contrary; then there exists a vertex $t+I$ adjacent to both $r+I$ and $s+I$ (hence $t+I,r+I,$ and $s+I$ are distinct elements of $R/I$). Then notice that $r,t,s$ are distinct, mutually adjacent vertices of $\Gamma_I(R)$. But this is a contradiction as $r\perp s$ in $\Gamma_I(R)$. Therefore $r+I\perp s+I$. Since $r+I\in V(\Gamma(R/I))$ was chosen arbitrarily, it follows that $\Gamma(R/I)$ is complemented.

``$\Leftarrow$'' Assume that $\Gamma(R/I)$ is complemented and $\sqrt{I}=I$. Since $\Gamma(R/I)$ is complemented and nonempty, it follows that $|V(\Gamma(R/I)|\geq 2$. Let $r\in V(\Gamma_I(R))$; then $r+I\in V(\Gamma(R/I))$. Since $\Gamma(R/I)$ is complemented, there exists a vertex $s+I$ in $\Gamma(R/I)$ such that $r+I \perp s+I$. Since these are vertices in $\Gamma(R/I)$, it follows that neither is zero in $R/I$; hence $r,s\not \in I$ and $rs\in I$. Thus $r$ and $s$ are adjacent vertices in $\Gamma_I(R)$. We claim that $r\perp s$ in $\Gamma_I(R)$. Assume to the contrary; then there exists a $t\in R\setminus I$ such that $r,s,$ and $t$ are distinct and mutually adjacent in $\Gamma_I(R)$. Using that $\sqrt{I}=I$, a similar argument to that in the forward implication shows that $r+I,s+I,$ and $t+I$ are distinct vertices of $\Gamma(R/I)$. It then follows that $r+I,s+I,$ and $t+I$ are distinct, mutually adjacent vertices of $\Gamma(R/I)$; but this is a contradiction as $r+I\perp s+I$. Therefore $r\perp s$ in $\Gamma_I(R)$. Since $r\in\Gamma_I(R)$ was chosen arbitrarily, it follows that $\Gamma_I(R)$ is complemented. 
\end{proof}

Combining the previous two theorems yields the following result.

\begin{corollary}

Let $R$ be a commutative ring with nonzero identity and $I$ a proper nonzero non-prime ideal of $R$. Then $\Gamma_I(R)$ is complemented if and only if exactly one of the following statements holds.
\begin{enumerate}
\item $R/I\cong \mathbb{Z}_4$ or $R/I \cong \mathbb{Z}_2[X]/(X^2)$, and $|I|=2$.
\item $\Gamma(R/I)$ is complemented and $I$ is a radical ideal of $R$.
\end{enumerate}
\end{corollary}

Using the fact that $R/I$ is reduced if and only if $\sqrt{I}=I$, we can extend the previous theorem to the following corollary using \cite[Theorem 3.5] {DFAetal2003}.Recall that if $I$ is a prime ideal, then all of the graphs in question are empty. We will consider the empty graph to be vacuously uniquely complemented.

\begin{corollary}\label{compcor1} Let $R$ be a commutative ring with nonzero identity and $I$ a radical ideal of $R$. Then the following statements are equivalent.
\begin{enumerate}
	\item $\Gamma_I(R)$ is complemented.
	\item $\Gamma(R/I)$ is complemented.
	\item $\Gamma(R/I)$ is uniquely complemented.
	\item $T(R/I)$ is von Neumann regular.
\end{enumerate}
\end{corollary}

We proceed to consider when $\Gamma_I(R)$ is uniquely complemented. Based on the preceding results, we are led to conjecture that when $I$ is a radical ideal, then  $\Gamma_I(R)$ is uniquely complemented if and only $\Gamma_I(R)$ is complemented. The following two lemmas are similar to those found in \cite[pp.~55-56]{redDis}. 

\begin{lemma} Let $R$ be a commutative ring with nonzero identity and $I$ a radical ideal of $R$. Then $x\perp y$ in $\Gamma_I(R)$ if and only if $x+I \perp y+I$ in $\Gamma(R/I)$.
\label{perplemma}
\end{lemma}

\begin{proof}
 Notice the lemma is vacuously true when $I=\{0\}$. Assume $I\neq \{0\}$.
``$\Rightarrow$'' First notice that $\sqrt{I}=I$ and $xy\in I$ implies that $x+I\neq y+I$. Otherwise, $y=x+i$ for some $i\in I$. Then $x^2=x(x+i)-xi=xy-xi\in I$. But $x\in V(\G_I(R))$ implies that $x\not \in I$. Hence $x\in \sqrt{I}$ and $x\not \in I$, but this is a contradiction as $\sqrt{I}=I$.

Also, $(x+I)(y+I)=0+I$, so that $x+I$ and $y+I$ are adjacent vertices of $\Gamma(R/I)$.  Assume to the contrary, that there exists $z+I\in V(\Gamma(R/I))$ such that $x+I-y+I-z+I-x+I$ is a triangle in $\Gamma(R/I)$. Then $x-y-z-x$ is a triangle in $\Gamma_I(R)$, which is a contradiction as $x\perp y$ in $\Gamma_I(R)$. Therefore, $x+I \perp y+I$ in $\Gamma(R/I)$ as desired.

``$\Leftarrow$'' Assume that $x+I \perp y+I$ in $\G(R/I)$. Then $xy\in I$; whence $x$ and $y$ are adjacent in $\G_I(R)$. Assume that $x\not \perp y$. Then there exists a vertex $c$ adjacent to both $x$ and $y$ in $\G_I(R)$. We claim that then $c+I$ is distinct from $x+I$ and $y+I$ and each of these three vecrtices are adjacent to each other. To see that $c+I$ is distinct from $x+I$ and $y+I$, assume to the contrary. Without loss of generality, assume $c+I=x+I$. Then $c=x+i$ for some $i\in I$. Then $cx\in I$ implies that $x^2\in I$, which is a contradiction as $\sqrt{I}=I$ and $x+I$ is nonzero. Since $x+I$, $y+I$, and $c+I$ are distinct and $xy$, $yc$, and $xc\in I$, it follows that $x+I$, $y+I$, and $c+I$ is a three-cycle in $\G(R/I)$. But this is a contradiction as  $x+I \perp y+I$ in $\G(R/I)$. \end{proof}

\begin{lemma} \label{unqcond} Let $R$ be a commutative ring with nonzero identity and $I$ a radical ideal of $R$. If $\Gamma(R/I)$ is uniquely complemented, $x\perp y$ and $x\perp z$ in $\Gamma_I(R)$, and $\alpha \in R\setminus I$, then
$$\alpha y \in I \text{ if and only if } \alpha z \in I.$$
\end{lemma}

\begin{proof}
The statement is symmetric in terms of $y$ and $z$; so it suffices to show that $\alpha y \in I \Rightarrow \alpha z \in I$.

By Lemma \ref{perplemma}, $x+I \perp y+I$ and $x+I \perp z+I$ in $\Gamma(R/I)$. Since $\Gamma(R/I)$ is uniquely complemented, it follows that $\text{ann}_{R/I}(y+I)=\text{ann}_{R/I}(z+I)$ (here we also using the fact $\text{ann}_{R/I}(y+I)\setminus\{y+I\}=\text{ann}_{R/I}(y+I)$ and $\text{ann}_{R/I}(x+I)\setminus\{x+I\}=\text{ann}_{R/I}(x+I)$ since $\sqrt{I}=I$ ).

Assume $\alpha y \in I$. 
 Then $\alpha+I \in \text{ann}_{R/I}(y+I)=\text{ann}_{R/I}(z+I)$. Hence $(\alpha+I)(z+I)=0+I$, and therefore $\alpha z \in I$ as desired.
\end{proof}

\begin{theorem} \label{radicalG_Iuc}
Let $R$ be a commutative ring with nonzero identity and $I$ a radical ideal of $R$. Then $\Gamma_I(R)$ is complemented if and only if $\Gamma_I(R)$ is uniquely complemented.
\end{theorem}

\begin{proof}
If $I=(0)$, then the result follows from \cite[Theorem 3.5]{DFAetal2003}. If $\G_I(R)$ is the empty graph, the statement holds vacuously. Assume that $I\neq (0)$ and that $\G_I(R)$ is not the empty graph (i.e., $I$ is not a prime ideal of $R$).
 
The reverse implication is by definition.

Assume $\Gamma_I(R)$ is complemented. Then $\Gamma_I(R)$ has at least two elements, and thus $V(\Gamma(R/I))$ must be nonempty. Since $I$ is a radical ideal, it follows that $|V(\Gamma(R/I))|\neq 1$ (since there are only two rings up to isomorphism with exactly 2 zero-divisors, and they are both non-reduced rings). Thus $|V(\Gamma(R/I))|\geq 2$, and hence $\Gamma(R/I)$ is complemented by Theorem \ref{comptheorem}. Moreover, $\Gamma(R/I)$ is uniquely complemented by Corollary \ref{compcor1}. The desired result then follows from Lemma \ref{unqcond}. \end{proof}

\begin{theorem}Let $R$ be a commutative ring with nonzero identity and $I$ a proper radical ideal of $R$. Then the following statements are equivalent.
\begin{enumerate}
	\item $\Gamma_I(R)$ is complemented.
	\item $\Gamma_I(R)$ is uniquely complemented.
	\item $\Gamma(R/I)$ is complemented.
	\item $\Gamma(R/I)$ is uniquely complemented.
	\item $T(R/I)$ is von Neumann regular.
\end{enumerate}
Moreover, regardless if $I$ is a radical or non-radical ideal, $\Gamma_I(R)$ is complemented if and only if $\Gamma_I(R)$ is uniquely complemented.
\end{theorem}

\begin{proof}
If $I$ is a prime ideal ideal of $R$, then all of the graphs in question are empty and $R/I$ is an integral domain. Thus all of the conditions hold. 

If $I=(0)$ and radical, then the theorem  holds by \cite[Theorem 3.5]{DFAetal2003}; in this case, the conditions (1) and (3) are equivalent as are conditions (2) and (4).

Assume that $I$ is a nonzero, proper, non-prime, radical ideal of $R$. The equivalences follow from Corollary \ref{compcor1} and Theorem \ref{radicalG_Iuc}. 

For the ``moreover statement,'' if $I$ is not a radical ideal, then $\Gamma_I(R)$ is complemented if and only if $\Gamma_I(R)\cong K^2$ by Theorem \ref{notradicalK2}. However, $K^2$ is uniquely complemented. Thus, regardless of whether or not $I$ is a radical ideal of $R$, we have $\Gamma_I(R)$ is uniquely complemented if and only if $\Gamma_I(R)$ is complemented. \end{proof}

{\bf Acknowledgment:} I would like to thank my advisor, David F. Anderson, for his contribution and comments in my graduate research. This material is derived from dissertation research performed at the University of Tennessee, Knoxville \cite{smithDis}.

\bibliography{mybibfile}

\end{document}